\documentclass[11pt]{amsart}
\usepackage[T1]{fontenc}
\usepackage[utf8]{inputenc}
\usepackage{amsmath,amssymb,amsthm}
\usepackage{mathtools}
\usepackage{booktabs}
\usepackage[margin=1.1in]{geometry}
\usepackage{hyperref}

\DeclareMathOperator{\rad}{rad}

\theoremstyle{plain}
\newtheorem{theorem}{Theorem}[section]
\newtheorem{proposition}[theorem]{Proposition}
\newtheorem{lemma}[theorem]{Lemma}
\newtheorem{corollary}[theorem]{Corollary}
\newtheorem{conjecture}[theorem]{Conjecture}

\theoremstyle{definition}
\newtheorem{definition}[theorem]{Definition}

\newtheorem{question}[theorem]{Question}

\theoremstyle{remark}
\newtheorem{remark}[theorem]{Remark}

\numberwithin{equation}{section}

\begin{document}

\title{Weak Hall's conjecture, approximation gains, and continued fractions}

\author{R. Laniewski}
\author{K. M\"uller}

\date{September 28, 2026}



\begin{abstract}\label{abs:hall}
For an integer solution of $y^2=x^3+k$ with $x,y\ge 1$ and $k\ne 0$, the integers $x^3$, $|k|$ and $y^2$ form a triple in which one term is the sum of the other two. When $\gcd(x,y)=1$ this triple is coprime, and the $abc$ conjecture predicts that its quality is asymptotically at most $1$. Following M\"uller, Taktikos and de Weger, we factor this quality through the product $P=xy|k|$ into the approximation gain $G_a=\log\max(x^3,y^2)/\log P$ and the power gain $G_p=\log P/\log\rad(xy|k|)$, which is at least $1$. We show that the weak form of Hall's conjecture, which asks that $|k|>x^{1/2-\delta}$ for every $\delta>0$ outside finitely many solutions, is equivalent to the asymptotic bound $G_a\le 1$. More generally, for $6/11\le\kappa<6/5$, an asymptotic bound $G_a\le\kappa$ gives $|k|>x^{3/\kappa-5/2-\delta}$ for every $\delta>0$ outside finitely many solutions, and for $3/4\le\kappa<6/5$ the two statements are equivalent. Since the approximation gain never exceeds the quality, an $abc$ inequality with exponent $1\le\kappa<6/5$ gives the same bound for primitive solutions. The case $\kappa=1$ extends to all solutions and recovers the classical implication from the $abc$ conjecture to weak Hall. We also relate small values of $|k|$ to large partial quotients of $\sqrt x$.
\end{abstract}

\maketitle

\section{Introduction}\label{sec:intro}

Hall's conjecture concerns the integer solutions of the Mordell equation $y^2=x^3+k$ with $k\ne 0$. Its original form asks for a constant $C>0$ with $|k|>C\sqrt x$ for all solutions~\cite{Hall1971}, and Danilov exhibited infinitely many solutions with $|k|<0.97\sqrt x$~\cite{Danilov1982}. The weak form allows a loss of $x^{\delta}$.

\begin{conjecture}[Weak Hall]\label{conj:weakhall}
For every $\delta>0$, there are only finitely many solutions $(x,y,k)$ of $y^2=x^3+k$ with $x,y\ge 1$ and $0<|k|<x^{1/2-\delta}$.
\end{conjecture}

Conjecture~\ref{conj:weakhall} is equivalent to the more familiar formulation asking, for every $\delta>0$, for a constant $C(\delta)>0$ with $|k|\ge C(\delta)x^{1/2-\delta}$ for all solutions. Indeed, the finitely many exceptions allowed by the conjecture can be absorbed into the constant. Conversely, the bound $|k|\ge C(\delta/2)x^{1/2-\delta/2}$ gives $|k|\ge x^{1/2-\delta}$ as soon as $x^{\delta/2}\ge 1/C(\delta/2)$, and only finitely many solutions with $|k|<x^{1/2-\delta}$ have $x$ below this bound.

The implication from the $abc$ conjecture to Conjecture~\ref{conj:weakhall} is classical (see~\cite{Elkies2000}). M\"uller and Taktikos introduced a factorization of the quality of an $abc$-triple into an approximation gain and a power gain~\cite{MullerTaktikos2026}, and de Weger developed variants for all $abc$-triples~\cite{deWeger2026}. Following this principle, we choose the product $P=xy|k|$, adapted to the presentation of the Mordell equation as a difference of a square and a cube. The present note studies the gains defined by this choice and their relation to Hall's problem.

Theorem~\ref{thm:hallgain} shows that Conjecture~\ref{conj:weakhall} is equivalent to the statement that only finitely many solutions have approximation gain $G_a=\log\max(x^3,y^2)/\log P$ above $1+\varepsilon$. Proposition~\ref{prop:kappa} extends one direction to every level $\kappa$ in $[6/11,6/5)$, with the Hall exponent $3/\kappa-5/2$, and both directions to the levels $\kappa\ge 3/4$, a range that Remark~\ref{rem:threequarters} shows to be optimal. Since the approximation gain never exceeds the quality, an $abc$ inequality of exponent $\kappa$ gives this exponent at once (Corollary~\ref{cor:abcexponent}). The case $\kappa=1$ is the classical implication. For $\kappa\ge 6/5$ the exponent $3/\kappa-5/2$ is non-positive, and Remark~\ref{rem:szpiro} places the known consequences of Szpiro's conjecture for the quality with respect to this range. Remark~\ref{rem:normalized} compares the gains of $P$ with those of the product $x^2y|k|$ arising in the construction of M\"uller and Taktikos.

Section~\ref{sec:cf} relates small residuals to continued fractions. If $|k|$ is below roughly $\sqrt x$ and $\gcd(x,y)=1$, then $y/x$ is a convergent of $\sqrt x$, and $\sqrt x/|k|$ lies between $a/2$ and $(a+2)/2$ up to a factor tending to $1$, where $a$ is the next partial quotient. Section~\ref{sec:questions} states one question.

\section{Solutions and their gains}\label{sec:gains}

\begin{definition}\label{def:mordell}
A \emph{solution} is a triple $(x,y,k)$ of integers with $y^2=x^3+k$, $x\ge 1$, $y\ge 1$ and $k\ne 0$. It is \emph{primitive} when $\gcd(x,y)=1$. For a solution, put
\[
M=\max\{x^3,y^2\},\qquad P=xy|k|,\qquad N=\rad(xy|k|)
\]
and define the \emph{approximation gain} and the \emph{power gain}
\[
G_a=\frac{\log M}{\log P},\qquad G_p=\frac{\log P}{\log N}.
\]
When $x\ge 2$, let $c=\log|k|/\log x$ denote the \emph{Hall exponent} of the solution.
\end{definition}

The integer $P$ is at least $2$, since $x=y=1$ would force $k=0$. Likewise $N\ge 2$, since $N=1$ would force $x=y=|k|=1$, which contradicts $k\ne0$. Both gains are therefore well defined.

\begin{lemma}\label{lem:mordelltriple}
Let $(x,y,k)$ be a primitive solution. Then $x$, $y$ and $k$ are pairwise coprime, and the integers $x^3$, $|k|$ and $y^2$ form a coprime triple, one of which is the sum of the other two, with largest term $M$ and radical $N$. Its quality
\[
X=\frac{\log M}{\log N}
\]
satisfies
\begin{equation}\label{eq:product}
X=G_aG_p,\qquad G_p\ge 1,\qquad G_a\le X.
\end{equation}
\end{lemma}

\begin{proof}
A prime dividing $x$ and $k$ divides $y^2$, and a prime dividing $y$ and $k$ divides $x^3$, so the three integers are pairwise coprime. If $k>0$, then $y^2=x^3+k$ exceeds $x^3$ and $k$, and if $k<0$, then $x^3=y^2+|k|$ exceeds $y^2$ and $|k|$. The radical of $x^3y^2|k|$ is $N$. The product formula follows by cancelling $\log P$. Since $P$ is divisible by each prime factor of $xyk$, one has $N\le P$, whence $G_p\ge 1$ and $G_a\le X$.
\end{proof}

The product $P$ divides $x^3y^2|k|$ and has the same prime factors. These properties give the factorization~\eqref{eq:product}, in the spirit of~\cite{MullerTaktikos2026,deWeger2026}. Here the gains depend on the chosen presentation through $x^3$, $y^2$ and $|k|$, and need not coincide with the gains defined in those works. Other products with the same two properties give other pairs of gains with product $X$, as in Remark~\ref{rem:normalized}.

\section{Weak Hall as a bound on the approximation gain}\label{sec:hallgain}

\begin{theorem}\label{thm:hallgain}
Let $\mathcal S$ be a set of solutions. The following statements are equivalent.
\begin{enumerate}
\item[\textup{(i)}] For every $\delta\in(0,1/2)$, only finitely many solutions in $\mathcal S$ satisfy $|k|<x^{1/2-\delta}$.
\item[\textup{(ii)}] For every $\varepsilon>0$, only finitely many solutions in $\mathcal S$ satisfy $G_a>1+\varepsilon$.
\end{enumerate}
In particular, Conjecture~\ref{conj:weakhall} holds if and only if \textup{(ii)} holds for the set of all solutions.
\end{theorem}

\begin{proof}
Since $\gamma(1)=1/2$ in the notation of~\eqref{eq:gamma} below, the two implications are the case $\kappa=1$ of the two parts of Proposition~\ref{prop:kappa}.
\end{proof}

For $\kappa$ in the interval $[6/11,6/5)$, put
\begin{equation}\label{eq:gamma}
\gamma(\kappa)=\frac{3}{\kappa}-\frac52
\end{equation}
so that $\gamma$ decreases on this interval from $\gamma(6/11)=3$ and tends to $0$ as $\kappa$ tends to $6/5$, with $\gamma(1)=1/2$. Every solution satisfies $x\le M^{1/3}$, $y\le M^{1/2}$ and $|k|<M$, hence $P<M^{11/6}$ and
\begin{equation}\label{eq:gain-floor}
G_a>\frac{6}{11}.
\end{equation}
For a level $\kappa<6/11$, the hypothesis of Proposition~\ref{prop:kappa}\textup{(a)} would therefore hold only for finite sets $\mathcal S$.

\begin{proposition}\label{prop:kappa}
Let $\mathcal S$ be a set of solutions.
\begin{enumerate}
\item[\textup{(a)}] Let $6/11\le\kappa<6/5$. Assume that, for every $\varepsilon>0$, only finitely many solutions in $\mathcal S$ satisfy $G_a>\kappa+\varepsilon$. Then, for every $\delta\in(0,\gamma(\kappa))$, only finitely many solutions in $\mathcal S$ satisfy $|k|<x^{\gamma(\kappa)-\delta}$.
\item[\textup{(b)}] Let $3/4\le\kappa<6/5$. Conversely, assume that, for every $\delta\in(0,\gamma(\kappa))$, only finitely many solutions in $\mathcal S$ satisfy $|k|<x^{\gamma(\kappa)-\delta}$. Then, for every $\varepsilon>0$, only finitely many solutions in $\mathcal S$ satisfy $G_a>\kappa+\varepsilon$.
\end{enumerate}
\end{proposition}

\begin{proof}
For (a), write $\gamma=\gamma(\kappa)$ and assume that infinitely many solutions in $\mathcal S$ satisfy $|k|<x^{\gamma-\delta}$. For a fixed $x$, only finitely many $y$ satisfy $|y^2-x^3|<x^{\gamma}$, so $x$ tends to infinity along a subsequence of these solutions. Since $\gamma-\delta<3$, such a solution has $|k|<x^3$, hence $y^2\le x^3+|k|\le 2x^3$ and $P\le\sqrt2\,x^{5/2+\gamma-\delta}$, while $M\ge x^3$. Therefore
\[
G_a\ge\frac{3\log x}{(5/2+\gamma-\delta)\log x+\log\sqrt2}
\]
and the right-hand side tends to $3/(5/2+\gamma-\delta)$ as $x\to\infty$, which exceeds $3/(5/2+\gamma)=\kappa$. The assumption therefore fails for $\varepsilon=\frac12\bigl(3/(5/2+\gamma-\delta)-\kappa\bigr)$.

For (b), assume that infinitely many solutions in $\mathcal S$ satisfy $G_a>\kappa+\varepsilon$. Replacing $\varepsilon$ by a smaller positive number, we may assume $\kappa+\varepsilon<6/5$. If $|k|\ge M/2$, then $P\ge x|k|\ge M^{4/3}/2$ when $M=x^3$, and $P\ge y|k|\ge M^{3/2}/2\ge M^{4/3}/2$ when $M=y^2$. In both cases
\[
G_a\le\frac{\log M}{\frac43\log M-\log 2}
\]
and the right-hand side tends to $3/4<\kappa+\varepsilon$ as $M\to\infty$. Only finitely many solutions have bounded $M$, so all but finitely many of the solutions with $G_a>\kappa+\varepsilon$ satisfy $|k|<M/2$. For these, $\min\{x^3,y^2\}=M-|k|>M/2$, so that $x^3>M/2$ and $y^2>M/2$, and
\[
P>\Bigl(\frac M2\Bigr)^{1/3}\Bigl(\frac M2\Bigr)^{1/2}|k|=\Bigl(\frac M2\Bigr)^{5/6}|k|.
\]
Put $e=1/(\kappa+\varepsilon)-5/6$, which is positive and at most $1/2$, and note that $3e=\gamma(\kappa+\varepsilon)$. Combined with $P<M^{1/(\kappa+\varepsilon)}$ and $M<2x^3$, the last estimate gives
\[
|k|<2^{5/6}M^{e}<2^{5/6+e}x^{\gamma(\kappa+\varepsilon)}<3x^{\gamma(\kappa+\varepsilon)}.
\]
Since $M<2x^3$ and $|k|<M/2$, only finitely many of these solutions have a given $x$, so $x$ tends to infinity along a subsequence of them. Put $\delta=\frac12\bigl(\gamma(\kappa)-\gamma(\kappa+\varepsilon)\bigr)$, which is positive. For $x^{\delta}\ge 3$ the bound gives $|k|<x^{\gamma(\kappa)-\delta}$, and the assumption of (b) fails for this $\delta$.
\end{proof}

\begin{remark}\label{rem:threequarters}
The lower bound $3/4$ in part \textup{(b)} is optimal. The solutions $(x,y,k)=(n,1,1-n^3)$ with $n\ge 2$ are primitive and satisfy $M=n^3$ and $P=n(n^3-1)$, so that $G_a$ tends to $3/4$. For $6/11\le\kappa<3/4$, this set satisfies the hypothesis of \textup{(b)}, since $|k|=n^3-1$ exceeds $x^{\gamma(\kappa)-\delta}$ for $n$ large, while infinitely many of its members satisfy $G_a>\kappa+\varepsilon$ for $\varepsilon<3/4-\kappa$.
\end{remark}

\begin{remark}\label{rem:critical}
Along a sequence of solutions with $x\to\infty$ and $\log|k|/\log x\to c$, where $c<3$, one has $y=x^{3/2+o(1)}$, and therefore
\[
G_a\longrightarrow\frac{3}{5/2+c}.
\]
Under these hypotheses the limit of the approximation gain determines the limiting Hall exponent and conversely, since $c\mapsto 3/(5/2+c)$ maps $[0,3)$ bijectively onto $(6/11,6/5]$ with inverse $\gamma$ of~\eqref{eq:gamma}. In particular, the limit $1$ corresponds to the exponent $1/2$. The correspondence concerns limits along such sequences. It does not hold for individual solutions, since $(2,3,1)$ has Hall exponent $0$ and $G_a=\log 9/\log 6\approx 1.2263$. Nor does a limit of $G_a$ alone determine the exponent, since the solutions of Remark~\ref{rem:threequarters} have $G_a\to 3/4$ while their exponent tends to $3$, outside the hypothesis $c<3$.
\end{remark}

\section{Consequences of the abc conjecture and of Szpiro's conjecture}\label{sec:abc}

For $\kappa\ge 1$, we say that the \emph{$abc$ inequality with exponent $\kappa$} holds for a set of coprime triples when, for every $\varepsilon>0$, only finitely many of them have quality above $\kappa+\varepsilon$. The $abc$ conjecture is the case $\kappa=1$ for all coprime triples~\cite{Oesterle1988Fermat}.

\begin{corollary}\label{cor:abcexponent}
Let $1\le\kappa<6/5$, and assume the $abc$ inequality with exponent $\kappa$ for the triples of Lemma~\ref{lem:mordelltriple}. Then, for every $\delta\in(0,\gamma(\kappa))$, only finitely many primitive solutions satisfy $|k|<x^{\gamma(\kappa)-\delta}$. For $\kappa=1$, in particular, the $abc$ conjecture implies Conjecture~\ref{conj:weakhall} for primitive solutions.
\end{corollary}

\begin{proof}
Each such triple arises from only finitely many primitive solutions, since an assignment of its terms to $x^3$, $y^2$ and $|k|$ determines the solution. For instance, the solutions $(2,3,1)$ and $(1,3,8)$ give the same triple. The assumption therefore leaves, for every $\varepsilon>0$, only finitely many primitive solutions with $X>\kappa+\varepsilon$, and by~\eqref{eq:product} only finitely many with $G_a>\kappa+\varepsilon$. Proposition~\ref{prop:kappa} concludes.
\end{proof}

\begin{remark}\label{rem:nonprimitive}
If the $abc$ inequality with exponent $\kappa$ is assumed for all coprime triples, then the corollary extends to all solutions. Let $g=\gcd(x^3,y^2)$, which divides $k$. The integers $x^3/g$, $|k|/g$ and $y^2/g$ form a coprime triple with largest term $M/g$, since $x^3/g$ and $y^2/g$ are coprime and $|k|/g$ is their difference, and its radical divides $\rad(xy|k|)$. For a prime $p$ dividing $g$, with $\alpha=v_p(x)$ and $\beta=v_p(y)$, both $\alpha$ and $\beta$ are positive, and $k=y^2-x^3$ gives $v_p(k)\ge\min\{3\alpha,2\beta\}=v_p(g)$, so that
\[
v_p(P)-v_p(g)=\alpha+\beta+\bigl(v_p(k)-\min\{3\alpha,2\beta\}\bigr)\ge\alpha+\beta\ge 2.
\]
 For a prime $p$ dividing $xyk$ but not $g$, one has $v_p(P/g)=v_p(P)\ge 1$. Hence $\rad(xy|k|)$ divides $P/g$, and in particular $\rad(xy|k|)\le P/g$, and the assumed inequality gives $M/g\le C_\varepsilon(P/g)^{\kappa+\varepsilon}$, that is $M\le C_\varepsilon P^{\kappa+\varepsilon}g^{1-\kappa-\varepsilon}\le C_\varepsilon P^{\kappa+\varepsilon}$, since $\kappa\ge 1$. For a prescribed $\tau>0$, choose $\varepsilon=\tau/2$. Then $G_a\le\kappa+\tau/2+\log C_{\tau/2}/\log P\le\kappa+\tau$ for sufficiently large $P$, and only finitely many solutions have bounded $P$. Proposition~\ref{prop:kappa}, applied to the set of all solutions, gives the conclusion of the corollary for all solutions, and for $\kappa=1$ the $abc$ conjecture gives Conjecture~\ref{conj:weakhall} in full. This extension matters, since the solution of Elkies with $\sqrt x/|k|\approx 46.6$~\cite{Elkies2000} has $\gcd(x,y)=3$.
\end{remark}

\begin{remark}\label{rem:szpiro}
The range $\kappa<6/5$ of Corollary~\ref{cor:abcexponent} is the range in which a quality bound gives a positive Hall exponent through $G_a\le X$. The known consequences of Szpiro's conjecture for the quality lie at its boundary and beyond. Assume Szpiro's conjecture in the form $|\Delta_E|\ll_\varepsilon N_E^{6+\varepsilon}$ for every $\varepsilon>0$ and every semistable elliptic curve $E$ over $\mathbb{Q}$, where $\Delta_E$ denotes the minimal discriminant and $N_E$ the conductor. This inequality gives the $abc$ inequality with exponent $3/2$ through the Frey curve, and with exponent $6/5$ through its quotients by subgroups of order $2$ (see~\cite[\S~I.2, Proposition~1 and the following remark, and \S~I.3]{Oesterle1988Fermat}).
\end{remark}

\begin{remark}\label{rem:normalized}
The quadratic approximation $y/x$ to $\sqrt x$ leads to the denominator product $P'=x^2y|k|$ in the construction of~\cite{MullerTaktikos2026}. This product also divides $x^3y^2|k|$ and has the same prime factors, and its power gain is related to that of $P$ by
\begin{equation}\label{eq:normalized-power}
G_p'=\frac{\log(x^2y|k|)}{\log N}=G_p+\frac{\log x}{\log N}.
\end{equation}
Along primitive solutions for which $y/x$ is a convergent of the irrational number $\sqrt x$, the $abc$ conjecture gives $G_p\to 1$ and $\log N\sim 3\log x$ by Question~\ref{q:gpa}, and therefore $G_p'\to 4/3$. For the product $P=xy|k|$, the Hall exponent is carried by the approximation gain, as Theorem~\ref{thm:hallgain} shows.
\end{remark}

\section{Small residuals and continued fractions}\label{sec:cf}

The following observation, due to the second author, describes the approximation side of the problem through the continued fraction of $\sqrt x$.

\begin{proposition}\label{prop:convergent}
Let $(x,y,k)$ be a primitive solution and put $u=(y+x^{3/2})/x$. If $|k|<u/2$, then $y/x$ is a convergent of the continued fraction of $\sqrt x$.
\end{proposition}

\begin{proof}
From $k=(y-x^{3/2})(y+x^{3/2})$ one obtains
\begin{equation}\label{eq:residual}
\Bigl|\sqrt x-\frac yx\Bigr|=\frac{|k|}{x(y+x^{3/2})}=\frac{|k|}{ux^2}.
\end{equation}
If $|k|<u/2$, then the right-hand side is smaller than $1/(2x^2)$. If $x$ is a perfect square, then this is impossible, since $\sqrt x$ is then an integer and $y/x$ differs from it by at least $1/x$. Otherwise $\sqrt x$ is irrational, the fraction $y/x$ is reduced since $\gcd(x,y)=1$, and Legendre's criterion~\cite{HardyWright2008} shows that it is a convergent.
\end{proof}

\begin{proposition}\label{prop:quotient}
Let $(x,y,k)$ be a primitive solution for which $y/x$ is a convergent $p_n/q_n$ of $\sqrt x$, let $a=a_{n+1}$ be the next partial quotient, and put $u=(y+x^{3/2})/x$ as in Proposition~\ref{prop:convergent}. Then
\begin{equation}\label{eq:quotient-bounds}
\frac{u}{a+2}<|k|<\frac ua.
\end{equation}
\end{proposition}

\begin{proof}
Write $\alpha=\sqrt x$ and let $\alpha_{n+1}$ be its next complete quotient. The convergent identity~\cite{HardyWright2008} gives
\begin{equation}\label{eq:complete-quotient-error}
\left|\alpha-\frac{p_n}{q_n}\right|=\frac{1}{q_n^2(\alpha_{n+1}+q_{n-1}/q_n)}.
\end{equation}
Here $a<\alpha_{n+1}<a+1$ and $0\le q_{n-1}/q_n\le 1$. Since $q_n=x$, equation~\eqref{eq:residual} gives~\eqref{eq:quotient-bounds}.
\end{proof}

Along these convergents with $x\to\infty$, the estimate $|\sqrt x-y/x|<x^{-2}$ gives $u=2\sqrt x+O(x^{-2})$. Proposition~\ref{prop:quotient} therefore places $\sqrt x/|k|$ between $a/2$ and $(a+2)/2$, up to a factor tending to $1$. A primitive solution with $|k|<x^{1/2-\delta}$ and $x$ sufficiently large has, by Propositions~\ref{prop:convergent} and~\ref{prop:quotient}, a next partial quotient $a>x^{\delta}$. Conversely, if $y/x$ is a convergent with $a\ge x^{\delta}$, then $|k|<3x^{1/2-\delta}<x^{1/2-\delta/2}$ for sufficiently large $x$. For primitive solutions, weak Hall is therefore equivalent to the statement that, for every $\delta>0$, only finitely many convergents $p_n/q_n$ of $\sqrt x$ with $q_n=x$ and $p_n=y$ are followed by a partial quotient $a\ge x^{\delta}$. For instance, the continued fraction of $\sqrt{5234}$ begins $[72;2,1,7,1,5,2,2,5,1,7,\ldots]$, and the solution $5234^3+17=378661^2$ arises from the convergent $p_9/q_9=378661/5234$, followed by the partial quotient $a_{10}=7$.

\section{A question}\label{sec:questions}

\begin{question}\label{q:gpa}
Among primitive solutions for which $y/x$ is a convergent of $\sqrt x$, how does the power gain $G_p$ vary with the partial quotient $a$? By Lemma~\ref{lem:mordelltriple} and Proposition~\ref{prop:quotient}, the quality satisfies
\[
X=G_aG_p\qquad\text{with}\qquad\log P=3\log x-\log a+O(1)
\]
so a large partial quotient lowers $\log P$ and raises $G_a$, while radical compression raises $G_p$. A relation between $a$ and $G_p$ along such solutions would describe how the two gains share the quality of Hall-extremal triples. Under the $abc$ conjecture both gains are constrained along these solutions. Since $|k|<u$ and $u=2\sqrt x+O(x^{-2})$, one has $P=O(x^3)$ and hence $G_a\ge 1-O(1/\log x)$, while $X\le 1+o(1)$. The product formula then gives $G_p=1+o(1)$, and $G_a\le 1+o(1)$ gives $\log a=o(\log x)$. Assuming the $abc$ conjecture, the question therefore concerns the rate at which these limits are approached.
\end{question}

\section*{Acknowledgments}\label{sec:acknowledgments}

The authors thank Olivier Rozier for suggesting the investigation of Hall's conjecture via this method, and Noam D. Elkies for providing valuable feedback on extremal examples.

\bibliographystyle{amsalpha}
\bibliography{hall_approximation_gain}

\end{document}